\documentclass[11pt]{article}
\usepackage{amsmath,amsfonts,amssymb,amsthm}
\usepackage{enumerate}
\usepackage{graphics}
\usepackage[all]{xy}

\theoremstyle{plain}
\newtheorem{thm}{Theorem}[section]
\newtheorem{cor}[thm]{Corollary}
\newtheorem{prop}[thm]{Proposition}

\newtheorem{lem}[thm]{Lemma}
\newtheorem{fact}[thm]{Fact}

\theoremstyle{definition}

\newtheorem{prob}[thm]{Problem}
\newtheorem{ex}[thm]{Example}
\newtheorem{conv}[thm]{Convention}
\newtheorem*{acknow}{Acknowledgement}

\theoremstyle{remark}
\newtheorem{rem}[thm]{Remark}

\newcommand{\R}{\mathbb{R}}
\newcommand{\C}{\mathbb{C}}
\newcommand{\F}{\mathbb{F}}
\newcommand{\Ha}{\mathbb{H}}

\newcommand{\g}{\mathfrak{g}}
\newcommand{\h}{\mathfrak{h}}
\newcommand{\frakk}{\mathfrak{k}}
\newcommand{\gc}{{\g_\C}}
\newcommand{\hc}{{\h_\C}}
\newcommand{\gu}{{\g_U}}
\newcommand{\hu}{{\h_U}}
\newcommand{\kg}{\mathfrak{k_g}}
\newcommand{\kh}{\mathfrak{k_h}}

\newcommand{\tgu}{\mathfrak{t}_{\gu}}
\newcommand{\thu}{\mathfrak{t}_{\hu}}
\newcommand{\tkg}{\mathfrak{t}_{\kg}}
\newcommand{\tkh}{\mathfrak{t}_{\kh}}

\newcommand{\fraku}{\mathfrak{u}}
\newcommand{\su}{\mathfrak{su}}
\newcommand{\so}{\mathfrak{so}}

\newcommand{\A}{\mathcal{A}}
\newcommand{\T}{\mathcal{T}}

\newcommand{\bl}{\bullet}
\newcommand{\bs}{\backslash}

\newcommand{\simto}{\overset{\sim}\to}

\newcommand{\rest}{\operatorname{rest}}
\newcommand{\im}{\operatorname{im}}

\newcommand{\rank}{\operatorname{rank}}

\DeclareMathOperator*{\cross}{\times}

\title{A topological necessary condition for 
the existence of compact Clifford--Klein forms}
\author{Yosuke Morita}
\date{}

\begin{document}
\maketitle

\begin{abstract}
We provide a necessary condition for the existence of a compact 
Clifford--Klein form of a given homogeneous space of reductive type. 
The key to the proof is to combine a result of Kobayashi--Ono with 
an elementary fact that certain two different 
Clifford--Klein forms have the same cohomology ring. 
We give some examples, $SL(p + q, \R)/SO(p, q) \ (p, q: \text{odd})$ 
for instance, of homogeneous spaces which do not admit compact 
Clifford--Klein forms. 
\end{abstract}

\section{Introduction}

Let $G$ be a Lie group and $H$ its closed subgroup. 
If a discrete subgroup $\Gamma$ of $G$ 
acts properly discontinuously and freely on $G/H$, 
the double coset space $\Gamma \bs G/H$ becomes a manifold 
locally modelled on $G/H$. 
The space $\Gamma \bs G/H$ is then called 
a Clifford--Klein form of $G/H$. 

In this paper, we study the following problem: 

\begin{prob}\label{prob:cptCK} (\cite{Kob89})
When does $G/H$ admit a compact Clifford--Klein form?
\end{prob}

Using the results of \cite{BHC} \cite{MosTam}, 
A. Borel \cite{Bor63} proved that when $G$ is linear reductive 
and $H$ is compact, $G/H$ always admits a compact Clifford--Klein form.
In contrast, if $G$ is a linear reductive Lie group
and $H$ is a non-compact closed reductive subgroup of $G$, 
$G/H$ does not necessarily admit a compact Clifford--Klein form. 
A systematic study of Problem~\ref{prob:cptCK} in this case
was initiated by Kobayashi \cite{Kob89}. 
Since then, various methods derived from diverse fields in mathematics 
have been applied to this problem (\cite{Ben} \cite{Kob92} \cite{KO} \cite{Mar} \cite{Sha} 
\cite{Zim} for instance). 
Methods and results on this topic are surveyed in 
Kobayashi \cite{Kob96-2} \cite{Kob05}, Kobayashi--Yoshino \cite{KY}, 
Labourie \cite{Lab} and Constantine \cite{Con}. 

\begin{ex}
If $(G,H) = (O(p,q+1), O(p,q))$ with $q \neq 1$, 
a compact Clifford--Klein form of $G/H$ is nothing but 
a compact complete pseudo-Riemannian manifold of signature $(p,q)$ 
with constant negative sectional curvature. 
When $q=0$ (Riemannian case), 
$G/H$ admits a compact Clifford--Klein form since $H$ is compact. 
In contrast, when $p$ and $q$ are odd, $G/H$ does not admit 
a compact Clifford--Klein form (see Kulkarni \cite{Kul} or 
Kobayashi--Ono \cite{KO}). 
Corollary~\ref{cor:nonsym} (4) 
combined with Fact~\ref{fact:connected} 
includes this result as a special case.
For more information on this example, see Kobayashi--Yoshino \cite{KY}.
\end{ex}

Extending the idea of Kobayashi--Ono \cite{KO} that 
$H^\bl(\Gamma \bs G/H)$ is ``larger than or equal to'' 
$H^\bl(G_U/H_U)$ 
if $\Gamma \bs G/H$ is a compact Clifford--Klein form, 
we obtain a topological obstruction for 
the existence of compact Clifford--Klein forms:

\begin{thm}\label{thm:main} 
{\rm (see Convention~\ref{conv:red} for notation and terminology)}
Let $G/H$ be a homogeneous space of reductive type, 
$G_U/H_U$ the compact homogeneous space associated to $G/H$ 
and $K_H$ the maximal compact subgroup of $H$. 
If the homomorphism
\[
\pi^\ast : H^\bl(G_U/H_U; \C) \to H^\bl(G_U/K_H; \C)
\]
induced by the projection $\pi : G_U/K_H \to G_U/H_U$ is not injective, 
then $G/H$ does not admit a compact Clifford--Klein form. 
\end{thm}

The key to the proof of Theorem~\ref{thm:main} is to combine 
the above idea of Kobayashi--Ono with an elementary fact that 
the cohomology rings of two different Clifford--Klein forms 
$H^\bl(\Gamma \bs G/H; \C)$ and $H^\bl(\Gamma \bs G/K_H; \C)$ 
are isomorphic to each other (see Proposition~\ref{prop:isom}). 

As an application of Theorem~\ref{thm:main}, 
we obtain some examples of symmetric spaces $G/H$ 
which do not admit compact Clifford--Klein forms: 

\begin{cor}\label{cor:list}
A symmetric space $G/H$ 
does not admit a compact Clifford--Klein form 
if $(G,H)$ is one of the following: 

\begin{enumerate}[(1)]
\item $(GL(2n,\R), GL(n,\C)) \quad (n>1)$ 
\item $(SL(p+q,\R), SO(p,q)) \quad (p,q : \text{odd})$
\item $(O(n,n), O(n,\C)) \quad (n>1)$ 
\item $(O(p+r, q+s), O(p,q) \times O(r,s)) 
\quad (p,q : \text{odd, } r>0)$
\end{enumerate}
\end{cor}

\begin{rem}
We mention some related results that were previously obtained by using different methods: 
\begin{itemize}
\item (1) is new to the best of the author's knowledge. 
\item Concerning (2), Kobayashi \cite{Kob92} proved that 
$SL(2p, \R) / SO(p, p)$ ($p > 0$) does not admit a compact Clifford--Klein form. 
Benoist \cite{Ben} gave alternative proof of this result, 
and also proved that $SL(2p+1, \R) / SO(p, p+1)$ ($p > 0$) does not admit a compact Clifford--Klein form. 
\item Concerning (3), Kobayashi \cite{Kob92} proved that 
$SO(n, n) / SO(n, \C)$ ($n : \text{even}$) does not admit a compact Clifford--Klein form. 
For odd $n$, (3) is new to the best of the author's knowledge. 
\item Concerning (4), Kobayashi \cite{Kob92} proved that 
$O(p+r, q+s) / (O(p,q) \times O(r,s))$ does not admit a compact Clifford--Klein form unless
$\min \{ p,q,r,s \} = 0$. We assume $s=0$ without loss of generality. Then, furthermore, 
$O(p+r,q) / (O(p,q) \times O(r))$ does not admit 
a compact Clifford--Klein form if $p+r > q$ (Kobayashi \cite{Kob92}), 
$(p,q,r) = (2n, 2n+1, 1)$ (Benoist \cite{Ben}), or
$p,q,r$ are all odd (Kobayashi--Ono \cite{KO}). On the other hand, it admits a compact Clifford--Klein form if 
$(p,q,r) = (1,2n,1)$, $(3,4n,1)$, $(7,8,1)$, $(1,4,3)$, $(1,4,2)$ (Kulkarni \cite{Kul}, Kobayashi \cite{Kob92} \cite{Kob96-2}).
\end{itemize}
\end{rem}

We can also apply our method to non-symmetric homogeneous spaces. 
For instance: 

\begin{cor}\label{cor:nonsym}
A homogeneous space $G/H$ 
does not admit a compact Clifford--Klein form 
if $(G,H)$ is one of the following: 
\begin{enumerate}[(1)]
\item $(SL(n_1+ \dots + n_k, \R), 
SL(n_1,\R) \times \dots \times SL(n_k,\R)) \quad (n_1, n_2 > 2)$ 
\item $(SL(n_1+ \dots + n_k, \C), 
SL(n_1,\C) \times \dots \times SL(n_k,\C)) \quad (n_1, n_2 > 1)$ 
\item $(SL(n_1+ \dots + n_k, \Ha), 
SL(n_1,\Ha) \times \dots \times SL(n_k,\Ha)) \quad (n_1, n_2 > 1)$ 
\item $(O(p_1+ \dots +p_k, q_1+ \dots +p_k), 
O(p_1,q_1) \times \dots \times O(p_k,q_k)) 
\quad (p_1, q_1 : \text{odd, } p_2>0)$ 
\end{enumerate}
\end{cor}

\begin{rem}
Corollary~\ref{cor:list} (4) is 
a special case of Corollary~\ref{cor:nonsym} (4). 
\end{rem}

\begin{rem}
\begin{itemize}
\item The existence problem of compact Clifford--Klein forms of $SL(n,\F)/SL(m,\F)$ ($n>m$, $\F = \R, \C, \Ha$) 
has been attracted considerable attentions. 
The first result was obtained in \cite{Kob92-2} in the setting $n=3, m=2, \F=\C$. 
Some further works are in \cite{Ben} \cite{Kob96-2} \cite{LMZ} \cite{LabZim} \cite{Sha} \cite{Zim}. 
For example, expanding the method of \cite{Zim} \cite{LMZ}, Labourie--Zimmer \cite{LabZim} proved that 
$SL(n, \R)/SL(m, \R)$ does not admit a compact Clifford--Klein form if $n-m > 2$. 
Unfortunately, Theorem~\ref{thm:main} gives no information about this case. 
\item Benoist \cite{Ben} proved that $SL(p+q, \R)/ (SL(p, \R) \times SL(q, \R))$ ($p,q > 0$) 
does not admit a compact Clifford--Klein form if $pq$ is even. 
\item By applying the method of \cite{Kob92}, Kobayashi \cite{Kob96-2} gave many results that are similar to Corollary~\ref{cor:nonsym}. See \cite[Example 4.13.5, Example 4.13.6, Example 4.13.7]{Kob96-2}. 
\end{itemize}
\end{rem}

\begin{rem}
We can also prove that 
$O(n_1+ \dots + n_k, \C) / 
(O(n_1,\C) \times \dots \times O(n_k,\C)) \ 
(n_1, n_2 > 1 \text{ or } n_1 : \text{even, } n_2=1)$ and
$Sp(n_1+ \dots + n_k, \C) / 
(Sp(n_1,\C) \times \dots \times Sp(n_k,\C)) \ (n_1, n_2 > 0)$ 
do not admit compact Clifford--Klein forms. 
However, these examples are not new. 
We can apply the method of \cite{Kob92} to these cases. 
\end{rem}

Kobayashi--Ono has already deduced a necessary condition 
for the existence of Clifford--Klein forms (\cite[Corollary 5]{KO}). 
Later, Kobayashi gave a generalization of this result (\cite[Proposition 4.10]{Kob89}). 
These methods hilighted the Euler class of tangent bundle. 
A feature of Theorem~\ref{thm:main} is that it includes information 
not only on the Euler class of tangent bundles
but also on other cohomology classes; 
to obtain the above examples, we use characteristic classes 
that are different from the Euler class of tangent bundles.
We give a proof of \cite[Proposition 4.10]{Kob89} 
in the spirit of \cite{KO} by using Theorem~\ref{thm:main} (see Corollary~\ref{cor:KO}). 

\section[Preliminaries and proof of Theorem 1.3]{Preliminaries and proof of Theorem~\ref{thm:main}}

We work in the following setting unless otherwise specified: 

\begin{conv}\label{conv:red}
$G$ is a linear reductive Lie group and 
$H$ is a closed connected subgroup of $G$ which is reductive in $G$. 
Without loss of generality, we shall realize $G$ and its subgroup $H$ 
as closed subgroups of $GL(N,\R)$ that are stable under transposition. 
$G_\C$ and $H_\C$ are connected Lie subgroups of $GL(N, \C)$ 
with Lie algebras $\gc = \g \otimes \C$ and $\hc = \h \otimes \C$, respectively. 
We assume that $G_\C$ and $H_\C$ are closed in $GL(n, \C)$. 
Put $G_U = G_\C \cap U(N)$ 
and $H_U = H_\C \cap U(N)$. They are compact connected real forms of $G_\C$ and $H_\C$, respectively. 
Finally, put $K_G = G \cap O(N)$ and $K_H = H \cap O(N)$. They are
maximal compact subgroups of $G$ and $H$, respectively. 
\end{conv}

\begin{rem}
(1) Since we assumed that $H$ is connected, $K_H$ is also connected by the Cartan decomposition, 
and hence a closed subgroup of $H_U$. 

(2) The assumption that $H$ is connected is not a serious restriction 
by the following result, 
which is essentially proved in Kobayashi \cite{Kob89}:
\end{rem}

\begin{fact}\label{fact:connected}
Let $G$ be a linear Lie group and $H$ a closed subgroup of $G$. 
Suppose $H$ has finitely many connected components. We denote by $H_o$ the 
identity component of $H$. 
Then, $G/H$ admits a compact Clifford--Klein form if and only if 
$G/H_o$ admits a compact Clifford--Klein form. 
\end{fact}

We prepare some results on the topology of Clifford--Klein forms 
to prove Theorem~\ref{thm:main}. First, we observe that: 

\begin{lem}\label{lem:ob}
Let $G_1$ be a Lie group, $G_2$ a closed subgroup of $G_1$, 
$G_3$ a closed subgroup of $G_2$ and 
$\Gamma$ a discrete subgroup of $G_1$. 

\begin{enumerate}[(1)]
\item If $\Gamma$ acts properly discontinuously on $G_1/G_2$, 
it also acts properly discontinuously on $G_1/G_3$. 
\item If $\Gamma$ acts freely on $G_1/G_2$, 
it also acts freely on $G_1/G_3$. 
\item If the assumptions of (1) and (2) are satisfied, 
the projection
$\pi : \Gamma \bs G_1/G_3 \to \Gamma \bs G_1/G_2$
becomes a fibre bundle with typical fibre $G_2/G_3$. 
\end{enumerate}
\end{lem}

\begin{proof}
(1) This is a special case of \cite[Lemma 1.3 (1)]{Kob93}. 

(2) $\Gamma$ acts freely on $G_1/G_2$ if and only if 
$x (\Gamma - \{ 1\}) x^{-1} \cap G_2 = \emptyset$ 
for any $x \in G_1$. Thus the statement follows. 

(3) This follows immediately from (1) and (2). 
\end{proof}

Suppose $\Gamma \bs G/H$ is a Clifford--Klein form of $G/H$. 
Then it follows from Lemma~\ref{lem:ob} (1) (2) that 
$\Gamma \bs G/K_H$ is also a Clifford--Klein form.
We do not assume that $\Gamma \bs G/H$ is compact 
until the compactness is needed. 

\begin{prop}\label{prop:isom}
The projection 
$\pi : \Gamma \bs G/K_H \to \Gamma \bs G/H$
induces an isomorphism 
\[
\pi^\ast: H^\bl(\Gamma \bs G/H; \C) \simto H^\bl(\Gamma \bs G/K_H; \C).
\]
\end{prop}

\begin{proof}
By Lemma~\ref{lem:ob} (3) and the Cartan decomposition, the projection 
$\pi : \Gamma \bs G/K_H \to \Gamma \bs G/H$ is a fibre bundle 
with contractible typical fibre $H/K_H$. 
Thus the statement is an immediate consequence of
the Leray--Serre spectral sequence. 
\end{proof}

Next, let us recall a homomorphism $\eta$ constructed in
Kobayashi--Ono \cite{KO}. The space $\A^p(G/H)^{G}$ 
of $G$-invariant $p$-forms on $G/H$ 
is canonically isomorphic to 
$(\Lambda^p (\g/\h)^*)^H = (\Lambda^p (\g/\h)^*)^\h$ 
($H$ is connected). 
Likewise, $\A^p(G_U/H_U)^{G_U}$ is canonically isomorphic to 
$(\Lambda^p (\gu/\hu)^\ast)^\hu$.
The natural isomorphism
\[
(\Lambda^p (\gu/\hu)^\ast)^\hu \otimes \C \simeq 
(\Lambda^p (\gc/\hc)^\ast)^\hc \simeq 
(\Lambda^p (\g/\h)^\ast)^\h \otimes \C
\]
induces
\[
\eta : \A^p(G_U/H_U)^{G_U} \otimes \C \simto 
\A^p(G/H)^{G} \otimes \C \hookrightarrow 
\A^p(\Gamma \bs G/H) \otimes \C.
\]
Taking cohomology, we obtain 
\[
\eta : H^p(G_U/H_U; \C) \to H^p(\Gamma \bs G/H; \C).
\]

\begin{fact}\label{fact:inj} {\rm (see \cite[Proposition 3.9]{KO})}
If $\Gamma \bs G/H$ is compact, $\eta$ is injective. 
\end{fact}

Now we shall prove Theorem~\ref{thm:main}. 

\begin{proof}[Proof of Theorem~\ref{thm:main}.] 
$\eta : H^\bl(G_U/K_H; \C) \to H^\bl(\Gamma \bs G/K_H; \C)$ 
can be defined in the same way as above. By definition, the diagram
\[
\xymatrix{
H^\bl(G_U/H_U; \C) \ar[r]^{\eta} \ar[d]^{\pi^\ast} & 
H^\bl(\Gamma \bs G/H; \C) \ar[d]^{\pi^\ast} \\
H^\bl(G_U/K_H; \C) \ar[r]^{\eta} & 
H^\bl(\Gamma \bs G/K_H; \C)
}
\]
is commutative. By Proposition~\ref{prop:isom}, 
$\pi^\ast$ on the right-hand side is isomorphic. 
If $\Gamma \bs G/H$ is compact, 
then $\eta$ on the above is injective by Fact~\ref{fact:inj} and 
therefore $\pi^\ast$ on the left-hand side has to be injective. 
This completes the proof. 
\end{proof}

\section{The Chern--Weil homomorphism and non-injectivity}

To apply Theorem~\ref{thm:main}, 
we have to find examples of $G/H$ such that 
$\pi^\ast : H^\bl(G_U/H_U; \C) \to H^\bl(G_U/K_H; \C)$ 
is not injective. 
In this section, we give a sufficient condition 
for non-injectivity, which is easy to verify in typical cases.

$\pi : G_U \to G_U/H_U$ is a principal $H_U$-bundle. 
Thus the Chern--Weil characteristic homomorphism 
\[
w : (S^p (\hu)^\ast)^{H_U} \to H^{2p}(G_U/H_U; \R) 
\subset H^{2p}(G_U/H_U; \C)
\]
is defined. It is straightforward to see that the diagram 
\[
\xymatrix{
(S (\hu)^\ast)^{H_U} \ar[r]^{w \ \ \ } \ar[d]^{\rest} & 
H^\bl(G_U/H_U; \C) \ar[d]^{\pi^\ast} \\
(S (\kh)^\ast)^{K_H} \ar[r]^{w \ \ \ } & 
H^\bl(G_U/K_H; \C)
}
\]
is commutative. Here, 
$\rest : (S (\hu)^\ast)^{H_U} \to (S (\kh)^\ast)^{K_H}$ is
the restriction map. 

\begin{fact}\label{fact:kerw} {\rm (see \cite[\S 10]{Car50-2})}
\[
\ker \left( w : (S (\hu)^\ast)^{H_U} \to H^\bl(G_U/H_U; \C) \right)
\]
is equal to the ideal $J_{G_U/H_U}$ generated by 
\[
\bigoplus_{p=1}^\infty 
\im \left( \rest : (S^p (\gu)^\ast)^{G_U} 
\to (S^p (\hu)^\ast)^{H_U} \right) .
\]
\end{fact}

\begin{prop}\label{prop:noninj}
Assume that 
\[
\ker \left( \rest : (S (\hu)^\ast)^{H_U} 
\to (S (\kh)^\ast)^{K_H} \right) \not\subset J_{G_U/H_U}, 
\]
where $J_{G_U/H_U}$ is as in Fact~\ref{fact:kerw}.
Then the homomorphism 
$\pi^\ast: H^\bl(G_U/H_U; \C) \to H^\bl(G_U/K_H; \C)$ 
induced by the projection $\pi : G_U/K_H \to G_U/H_U$ is not injective, 
and hence $G/H$ does not admit a compact Clifford--Klein form. 
\end{prop}

\begin{proof}
By Fact~\ref{fact:kerw}, 
we can pick $P \in (S (\hu)^\ast)^{H_U}$ such that 
$w(P) \neq 0$ and $P |_{\kh} = 0$. 
Then $w(P) \in H(G_U/H_U; \C)$ is a non-zero element of 
a kernel of $\pi^\ast : H(G_U/H_U; \C) \to H(G_U/K_H; \C)$. 
\end{proof}

By Chevalley's restriction theorem, 
we can rewrite Proposition~\ref{prop:noninj} 
in terms of Cartan subalgebras and Weyl groups as follows. 

\begin{conv}
We take maximal tori $T_{G_U}$ of $G_U$, $T_{H_U}$ of $H_U$, 
$T_{K_G}$ of $K_G$ and $T_{K_H}$ of $K_H$ such that 
$T_{G_U} \supset T_{H_U} \supset T_{K_H}$ and 
$T_{K_G} \supset T_{K_H}$. 
Their Lie algebras and their Weyl groups are denoted by 
$\tgu$, $\thu$, $\tkg$, $\tkh$, 
$W_{G_U}$, $W_{H_U}$, $W_{K_G}$ and $W_{K_H}$, respectively. 
\end{conv}

Let us denote by $I_{G_U/H_U}$ an ideal of $(S (\thu)^\ast)^{W_\hu}$ 
generated by 
\[
\bigoplus_{p=1}^\infty 
\im\left( 
\rest : (S^p (\tgu)^\ast)^{W_{G_U}} \to (S^p (\thu)^\ast)^{W_{H_U}} 
\right). 
\]
In other words, $I_{G_U/H_U}$ is 
the ideal of $(S (\thu)^\ast)^{W_{H_U}}$
corresponding to $J_{G_U/H_U}$ under the isomorphism
$(S (\thu)^\ast)^{W_{H_U}} \simeq (S (\hu)^\ast)^{H_U}$.

\begin{cor}\label{cor:noninj2}
Assume that 
\[
\ker\left( 
\rest : (S (\thu)^\ast)^{W_{H_U}} \to (S (\tkh)^\ast)^{W_{K_H}} 
\right)
\not\subset I_{G_U/H_U}. 
\]
Then $G/H$ does not admit a compact Clifford--Klein form. 
\end{cor}

\section{Reduction to maximal tori}

In the statement of Theorem~\ref{thm:main}, one can replace 
$K_H$ by $T_{K_H}$: 

\begin{cor}\label{cor:main2}
If the homomorphism
\[
\pi^\ast: H^\bl(G_U/H_U; \C) \to H^\bl(G_U/T_{K_H}; \C)
\]
induced by the projection $\pi : G_U/T_{K_H} \to G_U/H_U$ is not injective, 
is not injective, then $G/H$ does not admit a compact Clifford--Klein form. 
\end{cor}

In order to prove Corollary~\ref{cor:main2}, 
we use the following fact: 

\begin{fact}\label{fact:torus} 
{\rm (see \cite[Th\'eor\`eme 2.2]{Ler}, 
\cite[Theorem 6.8.3]{GuiSte})}
Let $K$ be a connected compact Lie group, 
$T$ a maximal torus of $K$ and $W$ its Weyl group. 
If $M$ is a manifold on which $K$ acts freely, 
the homomorphism
\[
\pi^\ast : H^\bl(M/K; \C) \to H^\bl(M/T; \C)
\]
induced by the projection $\pi : M/T \to M/K$ 
is injective and its image is $H^\bl(M/T; \C)^W$.
\end{fact}

\begin{proof}[Proof of Corollary~\ref{cor:main2}.]
Putting $M=G_U$ and $K=K_H$ in Fact~\ref{fact:torus}, we obtain that 
\[
\pi^\ast : H^\bl(G_U/K_H; \C) \to H^\bl(G_U/T_{K_H}; \C)
\]
is injective. Thus 
$\pi^\ast : H^\bl(G_U/H_U; \C) \to H^\bl(G_U/K_H; \C)$ is injective 
if and only if 
$\pi^\ast : H^\bl(G_U/H_U; \C) \to H^\bl(G_U/T_{K_H}; \C)$ 
is injective. Now the statement follows from Theorem~\ref{thm:main}.
\end{proof}

\section{Examples}

In this section, 
we prove Corollary~\ref{cor:list} and Corollary~\ref{cor:nonsym}. 

\begin{proof}[Proof of Corollary~\ref{cor:list}.] 
(1) It is enough to confirm that 
the assumption of Proposition~\ref{prop:noninj} is satisfied when 
$G_U = U(2n)$, $H_U = U(n) \times U(n)$ and 
\[
K_H = 
\left\{ 
\begin{pmatrix}A&0 \\ 0&\bar{A}\end{pmatrix} : A \in U(n)
\right\}.
\]
Recall that $(S (\fraku(n))^\ast)^{U(n)}$ is the polynomial algebra generated by $\{c_1, \dots, c_n \}$, 
where $c_i \in (S^i (\fraku(n))^\ast)^{U(n)}$ refers 
the elementary symmetric polynomial of $U(n)$ of degree $i$. 
Geometrically, $c_i$ corresponds to the $i$-th Chern class of $U(n)$. 
Now, $(S (\gu)^\ast)^{G_U}$, $(S (\hu)^\ast)^{H_U}$ and $(S (\kh)^\ast)^{K_H}$
are the polynomial algebras generated by $\{ c_1, \dots, c_{2n} \}$, 
$\{c_1\otimes 1, \dots, c_n\otimes 1, 1\otimes c_1, \dots, 1\otimes c_n\}$ 
and $\{ c_1, \dots, c_n \}$, respectively. The restriction maps are given by
\[
\rest : (S (\gu)^\ast)^{G_U} \to (S (\hu)^\ast)^{H_U}, \quad 
c_i \mapsto c_i \otimes 1 + c_{i-1} \otimes c_1 + \dots + 1 \otimes c_i
\]
and
\[
\rest : (S (\hu)^\ast)^{H_U} \to (S (\kh)^\ast)^{K_H}, \quad 
c_i \otimes 1 \mapsto c_i, \ 1 \otimes c_i \mapsto (-1)^i c_i.
\]
Therefore, 
\[
c_2 \otimes 1 - 1 \otimes c_2 \in 
\ker \left( \rest : (S (\hu)^\ast)^{H_U} 
\to (S (\kh)^\ast)^{K_H} \right)
\]
On the other hand,
\[
c_2 \otimes 1 - 1 \otimes c_2 \notin J_{G_U/H_U}, 
\]
namely, $c_2 \otimes 1 - 1 \otimes c_2$ is not contained in the ideal of $(S (\hu)^\ast)^{H_U}$ 
generated by the restrictions of the positive-degree parts of the elements of $(S (\gu)^\ast)^{G_U}$. 
Thus the assumption of Proposition~\ref{prop:noninj} is satisfied. 

(2) We first remark that 
we may replace $SO(p,q)$ with its identity component $SO_o(p,q)$ 
by Fact~\ref{fact:connected}. 
Thus it suffices to confirm that 
the assumption of Proposition~\ref{prop:noninj} is satisfied when 
$G_U = SU(p+q)$, $H_U = SO(p+q)$ and $K_H = SO(p) \times SO(q)$.
Recall that $(S (\so(n))^\ast)^{SO(n)}$ 
is the polynomial algebra generated by $\{ p_1, \dots, p_{\frac{n-1}{2}} \}$ when $n$ is odd, 
and by $\{ p_1, \dots, p_{\frac{n-2}{2}}, e \}$ when $n$ is even (note that $p_{\frac{n}{2}} = e^2$). 
Here $p_i \in (S^{2i} (\so(n))^\ast)^{SO(n)}$ corresponds to the $i$-th Pontrjagin class 
and $e \in (S^{\frac{n}{2}} (\so(n))^\ast)^{SO(n)}$ corresponds to the Euler class. 
Since $p+q$ is even, 
\[
(S (\hu)^\ast)^{H_U} = (S (\so(p+q))^\ast)^{SO(p+q)} 
\]
is freely generated by $\{ p_1, \dots, p_{\frac{p+q-2}{2}}, e \}$. 
Now, since $p$ and $q$ are odd, 
\[
e \in \ker \left( \rest : (S (\hu)^\ast)^{H_U} 
\to (S (\kh)^\ast)^{K_H} \right), 
\] 
namely, the restriction of the Euler class $e$ to $\so(p) \oplus \so(q)$ is equal to zero. 
On the other hand, since the restrictions of the elements of $(S (\gu)^\ast)^{G_U}$ are written as 
polynomials of Pontrjagin classes, $e \notin J_{G_U/H_U}$. 

(3) (4) The proofs are analogous to (1) and (2); we consider 
$p_1 \otimes 1 - 1 \otimes p_1$ and $e \otimes 1$, respectively.
\end{proof}

Next, we shall prove Corollary~\ref{cor:nonsym}. 
We use the following general results:

\begin{prop}\label{prop:enlarge}
\begin{enumerate}[(1)]
\item Let $\widetilde{G}$ be a linear reductive Lie group, 
$G$ a closed subgroup of $\widetilde{G}$ 
and $H$ a closed connected subgroup of $G$. 
Assume that $G$ is reductive in $\widetilde{G}$ and $H$ is reductive in $G$. 
If $(G,H)$ satisfies the assumption of Proposition~\ref{prop:noninj} 
(or equivalently, Corollary~\ref{cor:noninj2}), 
so does $(\widetilde{G}, H)$. 
\item Let $G$ be a linear reductive Lie group. Let $H, H'$ be two 
closed connected subgroups of $G$ such that $H \cap H' = \{ 1 \}$ and $H' \subset Z(H)$. 
Assume that $H \times H'$ is reductive in $G$. 
If $(G,H)$ satisfies the assumption of Proposition~\ref{prop:noninj}, 
so does $(G, H \times H')$. 
\end{enumerate}
\end{prop}

\begin{proof}
(1) Without loss of generality, we may assume that 
$\widetilde{G}$, $G$ and $H$ are stable under transposition. 
Thus we can define $\widetilde{G}_U$ and $\tilde{\g}_U$ 
as in Convention~\ref{conv:red}. 
Since the restriction map
$(S(\tilde{\g}_U)^\ast)^{\widetilde{G}_U} \to (S(\hu)^\ast)^{H_U}$ 
factors $(S (\gu)^\ast)^{G_U}$, the image of 
\[
\rest : (S^p(\tilde{\g}_U)^\ast)^{\widetilde{G}_U} 
\to (S^p (\hu)^\ast)^{H_U}
\]
is contained in the image of 
\[
\rest : (S^p (\gu)^\ast)^{G_U} 
\to (S^p (\hu)^\ast)^{H_U}
\]
for each $p$. Hence $J_{\widetilde{G}_U/H_U} \subset J_{G_U/H_U}$ and
the statement follows. 

(2) Without loss of generality, 
we may assume that $G$, $H$ and $H'$ are stable under transposition. 
Thus we can define $H'_U$, $\h'_U$, $K_{H'}$ and $\frakk_{\h'}$ 
as in Convention~\ref{conv:red}. 
We remark that 
\[
(S(\hu \oplus \h'_U)^\ast)^{H_U \times H'_U} 
\simeq (S(\hu)^\ast)^{H_U} \otimes (S(\h'_U)^\ast)^{H'_U}.
\]
By the assumption of Proposition~\ref{prop:noninj} for $(G,H)$, 
there exists $P \in (S (\hu)^\ast)^{H_U}$ such that 
$P \notin J_{G_U/H_U}$ and $P |_{\kh} = 0$. Then 
$P \otimes 1 \in (S(\hu)^\ast)^{H_U} \otimes (S(\h'_U)^\ast)^{H'_U}$ 
satisfies $P |_{\kh \oplus \frakk_{\h'}} = 0$. 
Furthermore, $P \otimes 1 \notin J_{G_U/(H_U \times H_U')}$; 
it is a straightforward consequence of the fact that 
the restriction map $(S (\gu)^\ast)^{G_U} \to (S(\hu)^\ast)^{H_U}$
factors $(S(\hu)^\ast)^{H_U} \otimes (S(\h'_U)^\ast)^{H'_U}$. 
Thus $(G, H \times H')$ also satisfies 
the assumption of Proposition~\ref{prop:noninj}. 
\end{proof}

\begin{proof}[Proof of Corollary~\ref{cor:nonsym}.]
(1) Suppose $n_1, n_2 > 2$. 
\[
(SL(n_1+n_2,\R), SL(n_1,\R) \times SL(n_2,\R))
\]
satisfies the assumption of Proposition~\ref{prop:noninj}. Indeed, 
\[
c_3 \otimes 1 \in 
\ker \left( \rest : (S (\hu)^\ast)^{H_U} 
\to (S (\kh)^\ast)^{K_H} \right), 
\] 
while $c_3 \otimes 1 \notin J_{G_U/H_U}$. 
Here, $c_i \in (S^i(\su(p))^\ast)^{SU(p)}$ refers 
the $i$-th Chern class of $SU(p)$. Then 
\[
(SL(n_1 + \dots + n_k, \R), SL(n_1,\R) \times SL(n_2,\R))
\]
also satisfies the assumption by Proposition~\ref{prop:enlarge} (1) 
and thus so does 
\[
(SL(n_1 + \dots + n_k, \R), SL(n_1,\R) \times \dots \times SL(n_k,\R))
\]
by Proposition~\ref{prop:enlarge} (2). 
In particular, 
$SL(n_1 + \dots + n_k, \R) / 
(SL(n_1,\R) \times \dots \times SL(n_k,\R))$ 
does not admit a Clifford--Klein form. 

The proofs of (2)--(4) are parallel to that of (1). 
\end{proof}

\begin{rem}
More generally, if $n_1, n_2 > 2$, then
\begin{itemize}
\item $SL(n_1+n_2+n_3,\R)/(SL(n_1,\R)\times SL(n_2,\R) \times H')$ and
\item $SL(n_1+n_2+n_3,\R)/(S(GL(n_1,\R)\times GL(n_2,\R)) \times H')$
\end{itemize}
do not admit compact Clifford--Klein forms 
for any closed subgroup $H'$ of $SL(n_3,\R)$ 
such that $H'$ is reductive in $SL(n_3, \R)$ and the complexification of $H'$ is closed in $SL(n_3,\C)$. 
The proof is the same as that of Corollary~\ref{cor:nonsym} (1). 
The similar results also hold for (2)--(4).
\end{rem}

\section{Some remarks}

We give a proof of \cite[Proposition 4.10]{Kob89} 
in the spirit of \cite{KO} rather than \cite{Kob89} 
which uses an argument of spectral sequence: 

\begin{cor}\label{cor:KO}
If $\rank G = \rank H$ and $\rank K_G > \rank K_H$, 
$G/H$ does not admit a compact Clifford--Klein form.
\end{cor}

\begin{proof}
It is well-known that the Euler characteristic $\chi(G_U/H_U)$ of 
$G_U/H_U$ is non-zero if and only if 
$\rank G_U = \rank H_U$, which is equivalent to $\rank G = \rank H$. 
By the Gauss--Bonnet--Chern theorem, 
\[
\chi(G_U/H_U) = \int_{G_U/H_U} e(\T(G_U/H_U)), 
\]
where $\T(G_U/H_U)$ denotes the tangent bundle of $G_U/H_U$. 
Hence the Euler class 
$e(\T(G_U/H_U)) \in H^n(G_U/H_U; \C) \ (n = \dim G - \dim H)$ 
is non-zero when $\rank G = \rank H$. 
Thus, by Corollary~\ref{cor:main2}, it suffices to show that 
$\pi^\ast : H^\bl(G_U/H_U; \C) \to H^\bl(G_U/T_{K_H}; \C)$ 
sends $e(\T(G_U/H_U))$ to zero if $\rank K_G > \rank K_H$. 

First, we note that 
\[
\T(G_U/H_U) = G_U \cross_{H_U} (\gu / \hu)
\]
and hence 
\[
\pi^\ast \T(G_U/H_U) = G_U \cross_{T_{K_H}} (\gu / \hu).
\]
Now, 
$\gu/\hu = (\tkg/\tkh) \oplus (\tkg / \tkh)^\perp$ as a real unitary 
representation of $T_{K_H}$ (an inner product on $\gu/\hu$ is 
defined by the Killing form of $\gu$). 
Therefore
\[
\pi^\ast \T(G_U/H_U) = 
\left( G_U \cross_{T_{K_H}} (\tkg / \tkh) \right) \oplus
\left( G_U \cross_{T_{K_H}} (\tkg / \tkh)^\perp \right).
\]
$G_U \cross_{T_{K_H}} (\tkg / \tkh)$ 
is a trivial bundle because $T_{K_H}$ acts trivially on 
$\tkg / \tkh$. Its typical fibre $\tkg / \tkh$ 
has non-zero dimension since $\rank K_G > \rank K_H$. 
As a consequence, 
$e \left( G_U \cross_{T_{K_H}} (\tkg / \tkh) \right) = 0$ and 
\begin{align*}
\pi^\ast e(\T(G_U/H_U))
&= e(\pi^\ast \T(G_U/H_U)) \\
&= e\left( G_U \cross_{T_{K_H}} (\tkg / \tkh) \right)
e\left( G_U \cross_{T_{K_H}} (\tkg / \tkh)^\perp \right) \\
&= 0.
\end{align*}
\end{proof}

The following results exhibit limitations of our method: 

\begin{prop}
$H^\bl(G_U/H_U; \C) \to H^\bl(G_U/K_H; \C)$ is injective 
if either (1) or (2) is satisfied: 
\begin{enumerate}[(1)]
\item $\rank H = \rank K_H$. 
\item $G$ is a complexification of $H$.
\end{enumerate}
\end{prop}

\begin{proof}
(1) By the proof of Corollary~\ref{cor:main2}, 
it suffices to show that 
$\pi^\ast: H^\bl(G_U/H_U; \C) \to H^\bl(G_U/T_{K_H}; \C)$
is injective. By assumption, $T_{K_H}$ coincides with $T_{H_U}$. 
Hence the injectivity follows from Fact~\ref{fact:torus} 
by putting $M=G_U$ and $K=H_U$. 

(2) If $G$ = $H_\C$, the projection $\pi : G_U/K_H \to G_U/H_U$ is 
rewritten as 
\[
\pi : (H_U \times H_U)/\Delta K_H \to (H_U \times H_U)/\Delta H_U. 
\]
$\pi^\ast : H^\bl((H_U \times H_U)/\Delta H_U; \C) \to 
H^\bl((H_U \times H_U)/\Delta K_H; \C)$ is injective 
because a group manifold $(H \times H) / \Delta H$ 
admits a compact Clifford--Klein form 
(see Kobayashi \cite[Example 4.8]{Kob89}). 
\end{proof}

\begin{ex}
\begin{enumerate}[(1)]
\item Suppose $(G,H) = (U(p+r,q), U(p,q) \times U(r))$ 
with $p \geqslant q$ and $p,q,r>0$. 
Then $\rank H = \rank K_H$, 
but only finite subgroups of $G$ act properly discontinuously on $G/H$ 
(in particular, $G/H$ does not admit a compact Clifford--Klein form) 
by the Calabi--Markus phenomenon. See Kobayashi \cite{Kob89}. 
\item Suppose $(G,H) = (SL(n,\C), SL(n,\R))$ with $n>1$. 
Then $G$ is a complexification of $H$, 
but only finite subgroups of $G$ act properly discontinuously on $G/H$ 
by the Calabi--Markus phenomenon. 
\end{enumerate}
\end{ex}

\begin{acknow}
The author expresses his sincere gratitude to 
Professor Toshiyuki Kobayashi for his advice and encouragement. 
This work was supported by 
the Program for Leading Graduate Schools, MEXT, Japan.
\end{acknow}

\noindent
\textsc{Graduate School of Mathematical Science, \\ The University of Tokyo, \\
3-8-1 Komaba, Meguro-ku, \\ Tokyo 153-8914, Japan} \\
\textit{E-mail address}: \texttt{ymorita@ms.u-tokyo.ac.jp}


\begin{thebibliography}{99}
\bibitem{Ben}
Y. Benoist, 
Actions propres sur 
les espaces homog\`enes r\'eductifs. 
{\it Ann. of Math. (2)} {\bf 144} (1996), 315--347, 
MR1418901, Zbl 0868.22013.
\bibitem{Bor53}
A. Borel,
Sur la cohomologie des espaces fibr\'es principaux et 
des espaces homog\`enes de groupes de Lie compacts, 
{\it Ann. of Math. (2)} {\bf 57} (1953), 115--207, 
MR0051508, Zbl 0052.40001;
English translation in: 
{\it Topological library, Part 3: Spectral sequences in topology}, 
World Scientific (2012), 191--318, 
translated by V. P. Golubyatnikov, edited by S. P. Novikov and I. A. Taimanov, 
MR2976669, Zbl 1264.55002.
\bibitem{Bor63}
A. Borel, 
Compact Clifford--Klein forms of symmetric spaces, 
{\it Topology} {\bf 2} (1963), 111--122, 
MR0146301, Zbl 0116.38603.
\bibitem{BHC}
A. Borel and Harish-Chandra, 
Arithmetic subgroups of algebraic groups, 
{\it Ann. of Math. (2)} {\bf 75} (1962), 485--535, 
MR0147566, Zbl 0107.14804.
\bibitem{Car50-2}
H. Cartan, 
La transgression dans un groupe de Lie et 
dans un espace fibr\'e principal, 
{\it Colloque de Topologie (espaces fibr\'es), 
Bruxelles, 1950}, 
George Thone, Li\`ege (1951), 57--71, 
MR0042427, Zbl 0045.30701; 
reprinted in \cite{GuiSte}. 
\bibitem{Con}
D. Constantine, 
Compact Clifford--Klein forms --- geometry, topology and dynamics, 
arXiv:1307.2183 (2013), 
to appear in {\it Proceedings of the conference Geometry, 
Topology and Dynamics in Negative Curvature (Bangalore 2010)}, 
London Mathematical Society Lecture Notes Series.
\bibitem{GuiSte}
V. Guillemin and S. Sternberg, 
{\it Supersymmetry and Equivariant de Rham Theory}, 
Springer--Verlag, Berlin (1999), 
MR1689252, Zbl 0934.55007.
\bibitem{Kob89}
T. Kobayashi, 
Proper action on a homogeneous space of reductive type, 
{\it Math. Ann.} {\bf 285} (1989), 249--263, 
MR1016093, Zbl 0662.22008.
\bibitem{Kob92-2}
T. Kobayashi, 
Discontinuous groups acting on homogeneous spaces of reductive type, 
\textit{Representation theory of Lie groups and Lie algebras 
(Fuji-Kawaguchiko, 1990)}, 59--75, 
World Sci. Publ., River Edge, NJ, (1992), 
MR1190750, Zbl 1193.22010.
\bibitem{Kob92}
T. Kobayashi, 
A necessary condition for the existence of 
compact Clifford--Klein forms of homogeneous spaces of reductive type, 
{\it Duke Math. J.} {\bf 67} (1992), 653--664, 
MR1181319, Zbl 0799.53056. 
\bibitem{Kob93}
T. Kobayashi, 
On discontinuous groups acting on homogeneous spaces with noncompact isotropy subgroups, 
{\it J. Geom. Phys.} {\bf 12} (1993), 133--144, 
MR1231232, Zbl 0815.57029.
\bibitem{Kob96-2}
T. Kobayashi, 
Discontinuous groups and Clifford-Klein forms 
of pseudo-Riemannian homogeneous manifolds, 
\textit{Algebraic and analytic methods in representation theory 
(S$\o$nderborg, 1994)}, Perspectives in Mathematics, vol. 17, 
Academic Press, San Diego, CA, (1996), 99--165, 
MR1415843, Zbl 0899.43005.
\bibitem{Kob96}
T. Kobayashi, 
Criterion for proper actions on 
homogeneous spaces of reductive groups, 
{\it J. Lie Theory} {\bf 6} (1996), 147--163, 
MR1424629, Zbl 0863.22010.
\bibitem{Kob05}
T. Kobayashi, 
On discontinuous group actions on non-Riemannian homogeneous spaces, 
{\it S$\bar{u}$gaku} {\bf 57} (2005) 267--281 (in Japanese), 
MR2163672; 
English translation in: 
{\it Sugaku Expositions} {\bf 22} (2009), 1--19, 
translated by M. Reid, 
MR2503485.
\bibitem{KO} 
T. Kobayashi and K. Ono, 
Note on Hirzebruch's proportionality principle, 
{\it J. Fac. Sci. Univ. Tokyo Sect. IA Math.} {\bf 37} (1990), 71--87, 
MR1049019, Zbl 0726.57019.
\bibitem{KY}
T. Kobayashi and T. Yoshino, 
Compact Clifford--Klein forms of symmetric spaces --- revisited, 
{\it Pure Appl. Math. Q.} {\bf 1} (2005), 591--663, 
MR2201328, Zbl 1145.22011.
\bibitem{Kul}
R. Kulkarni, 
Proper actions and pseudo-Riemannian space forms, 
{\it Adv. in Math.} {\bf 40} (1981), 10--51, 
MR0616159, Zbl 0462.53041.
\bibitem{Lab}
F. Labourie, 
Quelques r\'esultats r\'ecents sur 
les espaces localement homog\`enes compacts,
{\it Manifolds and geometry (Pisa, 1993), 
Sympos. Math., XXXVI} (1996), 267--283, 
Cambridge Univ. Press, Cambridge, 
MR1410076, Zbl 0861.53053.
\bibitem{LMZ}
F. Labourie, S. Mozes and R. J. Zimmer, 
On manifolds locally modelled on non-Riemannian homogeneous spaces, 
\textit{Geom. Funct. Anal.} \textbf{5} (1995), 955--965, 
MR1361517, Zbl 0852.22011.
\bibitem{LabZim}
F. Labourie and R. J. Zimmer, 
On the non-existence of cocompact lattices
for $\textrm{SL}(n)/\textrm{SL}(m)$, 
{\it Math. Res. Lett.} {\bf 2} (1995), 75--77, 
MR1312978, Zbl 0852.22009.
\bibitem{Ler}
J. Leray, 
Sur l'homologie des groupes de Lie, des espaces homog\`enes 
et des espace fibr\'es principaux, 
{\it Colloque de Topologie (espaces fibr\'es), 
Bruxelles, 1950}, 
George Thone, Li\`ege (1951), 101--115, 
MR0041148, Zbl 0042.41801.
\bibitem{Mar}
G. Margulis, 
Existence of compact quotients of homogeneous spaces, 
measurably proper actions, and decay of matrix coefficients, 
{\it Bull. Soc. Math. France} {\bf 125} (1997), 447--456,
MR1605453, Zbl 0892.22009.
\bibitem{MosTam}
G. D. Mostow and T. Tamagawa, 
On the compactness of arithmetically defined homogeneous spaces, 
{\it Ann. of Math. (2)} {\bf 76} (1962), 446--463,
MR0141672, Zbl 0196.53201. 
\bibitem{Sha}
Y. Shalom, 
Rigidity, unitary representations of semisimple groups, and 
fundamental groups of manifolds with rank one transformation group, 
{\it Ann. of Math. (2)} {\bf 152} (2000), 113--182, 
MR1792293, Zbl 0970.22011.
\bibitem{Zim}
R. J. Zimmer, 
Discrete groups and non-Riemannian homogeneous spaces, 
\textit{J. Amer. Math. Soc.} \textbf{7} (1994), 159--168, 
MR1207014, Zbl 0801.22009.
\end{thebibliography}
\end{document}